\documentclass[10pt]{article}
\usepackage{amsmath}
\usepackage{amsfonts}
\usepackage{amssymb}
\usepackage{graphicx}
\usepackage{amsthm}
\usepackage{hyperref}
\usepackage{accents}
\usepackage[numbers,sort&compress,square]{natbib}

\usepackage[english]{babel}
\usepackage [autostyle, english = american]{csquotes}
\MakeOuterQuote{"}

\numberwithin{equation}{section}
\theoremstyle{plain}
\newtheorem*{theorem*}{Theorem}
\newtheorem*{lemma*}{Lemma}
\newtheorem{theorem}{Theorem}
\newtheorem{lemma}{Lemma}[section]

\newtheorem{proposition}[lemma]{Proposition}
\newtheorem{innercustomthm}{Theorem}
\newenvironment{customthm}[1]
  {\renewcommand\theinnercustomthm{#1}\innercustomthm}
  {\endinnercustomthm}

\theoremstyle{definition}

\newtheorem{remark}[lemma]{Remark}
\newtheorem{example}[lemma]{Example}
\newtheorem*{remark*}{Remark}
\newtheorem*{example*}{Example}
\newtheorem*{er*}{Examples and Remarks}

\def\V{\Vert}

\usepackage{geometry}
\begin{document}
	
\title{On Singular Vortex Patches, II: Long-time dynamics}
\author{Tarek M. Elgindi\textsuperscript{1} \and In-Jee Jeong\textsuperscript{2}}
\footnotetext[1]{Department of Mathematics, UC San Diego. E-mail: telgindi@ucsd.edu.}
\footnotetext[2]{Department of Mathematics, Korea Institute for Advanced Study. E-mail: ijeong@kias.re.kr}
\date{\today}

\maketitle

\begin{abstract}
In a companion paper \cite{SVP1}, we gave a detailed account of the well-posedness theory for singular vortex patches. Here, we discuss the long-time dynamics of some of the classes of vortex patches we showed to be globally well-posed in \cite{SVP1}. In particular, we give examples of time-periodic behavior, cusp formation in infinite time at an exponential rate, and spiral formation in infinite time. 
\end{abstract}

\tableofcontents 

\section{Introduction}

A vortex patch is a solution to the incompressible Euler equation: \[\partial_t\omega+u\cdot\nabla\omega=0\] \[u=\nabla^\perp \Delta^{-1}\omega\] which is of the form $\omega(x,t)={\bf 1}_{\Omega(t)}(x)$ with $\Omega(t)\subset\mathbb{R}^2$. Such solutions play an important role in the rigorous analysis of general solutions to the incompressible Euler equation as well as an important role in modeling, from hurricanes to Jupiter's Great Red Spot (\cite{Hur1},\cite{Mar},\cite{PD}). Consequently, the dynamical behavior of planar vortex patches has been considered by many authors for well over a century. The simplest example of a planar vortex patch is the "Rankine vortex" which is the stationary solution $\omega(x,t)={\bf 1}_{B_1(0)}(x)$. This solution was shown to be dynamically stable in certain topologies, while asymptotic stability is known to fail as we shall now observe. Indeed, a remarkable observation of Kirchoff is that if $\Omega(0)$ is the interior of an ellipse, then ${\bf 1}_{\Omega(t)}$ solves the incompressible Euler equation with $\Omega(t)=R_{\alpha(0)t}\Omega(0)$ where $R_{\theta}$ is just the linear transformation which rotates vectors counterclockwise by angle $\theta$ and $\alpha(0)$ is some constant depending on the initial ellipse $\Omega(0)$. In fact, as has been shown in works of Burbea \cite{Bu} and then many others (\cite{HMV,HM,HM2,CCG3}), there is a rich family of time-periodic vortex patch solutions near the base circular solution. We refer the interested readers to a recent development \cite{GSPSY,HM3} and references therein. Such vortex patches have a very regular behavior, but it is unclear whether such behavior is generic even in an asymptotic sense. 

In a different direction, it was observed in several numerical experiments and some theoretical works (\cite{CT,Maj,Dr,DrMc,CM1,CM2,Drit88,PuM}) that small deformations of the initial state ${\bf 1}_{B_1(0)}$ often lead to the boundary of the vortex patch forming filaments which spiral around the center of the patch. Even more drastic behavior was observed in other numerical simulations (\cite{But}) which indicated self intersection in finite time. Finite time self-intersection or pinching was then ruled out in important work of Chemin \cite{C} and then in works of Bertozzi-Constantin \cite{BeCo} and Serfati \cite{Ser1} in the case when the boundary of $\Omega(0)$ is initially smooth. These works, however, do not preclude very rapid formation of small scales. In other words, while a vortex patch cannot self intersect in finite time, it is still possible that a vortex patch self intersects or develops high oscillations in infinite time at a very rapid rate as $t\rightarrow\infty$ as the early numerical simulations seem to suggest. To our knowledge, the infinite-time phenomena described above (spiral formation and pinching) have not been rigorously confirmed, though some results have been obtained in spatial domains with solid boundaries when the vortex patch is allowed to touch the boundary.

The purpose of this work is to rigorously establish some of the phenomena described above and others. Specifically, our goal is to give examples of vortex patches that exhibit non-trivial dynamical behavior as $t\rightarrow\infty$ even though they remain "globally regular." Our first result concerns existence of purely rotating patches which are given by a union of infinite sectors. We prove some classification results on rotating sectors, and obtain in particular existence of purely rotating solutions with zero mean. Note that the existence of rotating patches with \textit{compact support} which have corners at the origin has been numerically observed in \cite{LF1}.

Our second result is the construction of a compactly supported vortex patch in $\mathbb{R}^2$ consisting of eight flower petals coming out of a center (asymptotically near ${\bf 0}$, $\Omega(0)$ looks like eight sectors emanating from the origin). Such a patch can be placed into a very natural global well-posedness class so that for all $t>0$, the local picture at the origin remains asymptotically like eight flower petals and there is propagation of regularity for all time. As $t\rightarrow\infty$, however, we see four of the eight petals being ejected from the origin at an exponential rate in time. That is, each of those pieces of the vortex patch cusp exponentially fast as $t\rightarrow\infty$. While we do not give any global information on the dynamics of the full patch, we are able to control the dynamics exactly at the origin, which reduces to an ODE system, to establish the cusp formation in infinite time. This ``localization'' phenomenon also happens in the case of 3D vortex patches with 3D corners, and based on this observation one can establish finite-time blow up of $\V \omega(t)\V_{L^\infty}$ in time (\cite{EJ3}). We mention that this type of blow-up result is similar to the one proved in \cite{EJB,EJE}.

Our third result is of a more global nature. We take now four identical flower petals intersecting at the origin so that $\Omega(0)$ is 4-fold symmetric. Once more, we have shown in \cite{SVP1} that such initial data can be placed in a natural global well-posedness class. Again, the dynamics exactly at the corner is governed precisely by a predetermined system of ODE; in this case, however, the ODE just dictates that the four corners rotate at a fixed speed depending only on the size of the corners. To make a meaningful statement on the long-time behavior of the vortex patch we need another, global, piece of information. Now we also stipulate that the vortex patch be very close to a disk (the symmetric difference should have small measure). Due to the (infinite-time) stability of the circular vortex patch which has been known for some time (\cite{WP}), we now also have that the velocity field of the vortex patch is close in a certain topology to that of the circular vortex patch. This implies that "most" points of the vortex patch (those away from the origin) are rotating counterclockwise with speed $1$. Now if one can ensure that there is a discrepancy between the rotation of the corner (whose speed can be determined a-priori) and most of the points, for all time, one would believe that the boundary of the vortex patch must develop a spiral in infinite time. Carrying this out rigorously requires us to employ a few analytic tools but also some basic topological methods, but the basic idea is to use both the local stability of the corner dynamics and the global stability of the disk to give spiral formation. While a spiral forms in this case, it remains an open problem whether the perimeter of the boundary of the patch actually becomes unbounded.

We emphasize that vortex patches utilized in the current work involve corners meeting at the origin in a symmetric fashion. The boundary of the patch can be infinitely smooth away from the origin. The work of Danchin \cite{Da} has established that if the patch boundary is smooth away from a closed set, this property propagates globally in time. Global persistence of certain cusp structures were obtained in \cite{Da2}. In the case of patches with a single isolated corner, we have shown in \cite{SVP1} that, roughly speaking, the corner structure cannot propagate continuously in time in general. We refer the interested readers to numerical simulations \cite{CS,CD}. On the other hand, if the corners meet at the origin with rotational symmetry (e.g. Figure \ref{fig:petal}), then we have global well-posedness; see the statements of Theorem \ref{mainthm:wellposedness} below for details. 

\subsection*{Outline of the paper}

The rest of this paper is organized as follows. First in Section \ref{sec:prelim}, we fix some notations and precisely state the central well-posedness result from the companion paper \cite{SVP1}. In particular, the ODE system satisfied by the corner angles will be the basis of all the main results in this work. Then in Sections \ref{sec:rot}, \ref{sec:cusp}, and \ref{sec:spiral}, we construct patch solutions which rotate with constant speed, cusp, and spiral, respectively. 

\subsection*{Acknowledgments} 

T.M. Elgindi was partially supported by NSF DMS-1817134. I.-J. Jeong has been supported by the POSCO Science Fellowship of POSCO TJ Park Foundation and the National Research Foundation of Korea(NRF) grant (No. 2019R1F1A1058486).

\section{Preliminaries}\label{sec:prelim}

In this section, we collect the notations, conventions, and recall the relevant main results of the companion work. 

\subsection*{Notations and definitions}
 
\begin{itemize} 
	\item For $\theta \in [0,2\pi)$, we let $R_{\theta}$ be the matrix of counterclockwise rotation around the origin by the angle $\theta$. We say that a set $\Omega \subset \mathbb{R}^2$ is $m$-fold (rotationally) symmetric if $R_{2\pi/m}(\Omega) = \Omega$. Similarly, a scalar function $f : \mathbb{R}^2 \rightarrow \mathbb{R}$ is $m$-fold symmetric if $f(x) = f(R_{2\pi/m}x)$ for any $x \in \mathbb{R}^2$. On the other hand, a vector field $v : \mathbb{R}^2 \rightarrow \mathbb{R}^2$ is $m$-fold symmetric if $v(R_{2\pi/m}x) = R_{2\pi/m} v(x)$.
	
	\item Given two angles $0 \le \theta_1 < \theta_2 < 2\pi$, we define the sector  \begin{equation*}
	\begin{split}
	S_{\theta_1,\theta_2} = \{ (r,\theta) : \theta_1 < \theta < \theta_2 \}.
	\end{split}
	\end{equation*}
	
	
	\item The classical H\"older spaces are defined as follows: for $0 < \alpha \le 1$ and an open set $U \subset \mathbb{R}^2$, \begin{equation*}
	\begin{split}
	\V f \V_{C^\alpha(\overline{U})} &= \V f\V_{L^\infty(U)} + \V f\V_{{C}_*^\alpha(\overline{U})} \\
	&= \sup_{x \in U} |f(x)| + \sup_{x \ne x'} \frac{|f(x) - f(x')|}{|x-x'|^\alpha}.
	\end{split}
	\end{equation*} 
	
	\item By a $C^{1,\alpha}$-diffeomorphism, we mean a uniformly bi-Lipschitz map $\Psi : \mathbb{R}^2 \rightarrow \mathbb{R}^2$ satisfying $\nabla\Psi, \nabla(\Psi^{-1})  \in C^{\alpha}(\mathbb{R}^2)$. 
	
	
	\item We reserve the letter $K$ for the Biot-Savart kernel \begin{equation*}
	\begin{split}
	K(x) = \frac{1}{2\pi} \frac{x^\perp}{|x|^2}. 
	\end{split}
	\end{equation*}  Convolution against $\nabla K$ is defined in the sense of principal value integration. 
	\item A point in $\mathbb{R}^2$ is denoted by $x = (x_1,x_2)$ or by $y = (y_1,y_2)$. Often we slightly abuse notation and consider polar coordinates $(r,\theta)$, where $r = |x|$ and $\theta = \arctan(x_2/x_1)$. 
	\item Given $x \in \mathbb{R}^2$ and $r > 0$, we define $B_x(r) = \{ y  \in \mathbb{R}^2 : |x-y| < r \}$.

\end{itemize}

In the following we shall fix some value of $0<\alpha<1$ (a specific choice of $\alpha$ does not make any differences) and suppress the dependence of constants on $\alpha$. Moreover, with the exception of Section \ref{sec:rot}, the strength of the vorticity will be normalized to be 1.

\subsection*{Main well-posedness results from \cite{SVP1}}

The main result of \cite{SVP1} states that for a patch with ``corners'' meeting symmetrically at a point (which can be taken to be the origin without loss of generality), the uniform H\"older regularity up to the corner propagates globally in time. Moreover, the angles of the corners and the region between corners satisfy a closed system of ODEs. We shall encode the uniform regularity with a $C^{1,\alpha}$-diffeomorphism $\Psi : \mathbb{R}^2 \rightarrow \mathbb{R}^2$ satisfying $\Psi(\bf{0}) = \bf{0}$ and $\nabla\Psi|_{\bf{0}} = \mathrm{Id}$. These properties will be always assumed whenever we use the term ``$C^{1,\alpha}$-diffeomorphism''. By a corner, we mean a region which can be (locally) mapped from the exact sector by a $C^{1,\alpha}$-diffeomorphism. 

\begin{customthm}{A}\label{mainthm:wellposedness}
	Consider $\omega_0 = {\bf{1}}_{\Omega_0}$ where $\Omega_0$ satisfies the following properties:  $\Omega_0$ is $m$-fold rotationally symmetric around the origin with $m \ge 3$; the boundary $\partial\Omega_0$ is $C^{1,\alpha}$-smooth away from the origin; and there exists some $r_0>0$ and a $C^{1,\alpha}$-diffeomorphism $\Psi_0$ such that \begin{equation}\label{eq:local_expression} 
		\begin{split}
		& \Psi_0(\Omega_0) \cap B_{\bf{0}}(r_0) =  \left(\bigcup_{k = 0}^{m-1} \bigcup_{i = 1}^N S_{ \beta_{i,0} + \frac{2\pi k}{m}, \beta_{i,0}  + \zeta_{i,0} + \frac{2\pi k}{m} }  \right)\cap B_{\bf{0}}(r_0)
		\end{split}
		\end{equation} where the sectors $S_{ \beta_{i,0} + \frac{2\pi k}{m}, \beta_{i,0} + \zeta_{i,0} + \frac{2\pi k}{m} }$ do not intersect with each other for $1 \le i \le N$ and $0 \le k \le m-1$.  
	
	Then, the corresponding patch solution $\Omega(t)$ enjoys the same properties for all $t > 0$, with some $C^{1,\alpha}$-diffeomorphism $ \Psi(t)$ and $r(t) > 0$; that is, \begin{equation*} 
	\begin{split}
	& \Psi(t)(\Omega(t)) \cap B_{\bf{0}}(r(t)) =  \left(\bigcup_{k = 0}^{m-1} \bigcup_{i = 1}^N S_{ \beta_{i}(t) + \frac{2\pi k}{m}, \beta_{i}(t) + \zeta_{i}(t) + \frac{2\pi k}{m} }  \right)\cap B_{\bf{0}}(r(t)) 
	\end{split}
	\end{equation*} for some non-overlapping sectors $S_{ \beta_{i}(t) + \frac{2\pi k}{m}, \beta_{i}(t) + \zeta_{i}(t) + \frac{2\pi k}{m} }$. 
	
	Moreover, the set of angles $\{ \beta_i,\zeta_i \}_{i=1}^N$ is determined for all $t > 0$ completely by the initial conditions $\{\beta_{i,0},\zeta_{i,0}\}_{i=1}^N$: introducing for convenience $\gamma_{i+\frac{1}{2}}(t):= \beta_{i+1}(t) - \beta_i(t) - \zeta_i(t) $ (separation angles), we have the following system: \begin{equation}\label{eq:ode0}
	\begin{split}
	\frac{d \beta_1}{dt}   =C_m' \sum_{i=1}^N \sin(\frac{m}{4}(2\beta_i + \zeta_i)) \sin(\frac{m}{4}\zeta_i ) - C_m'' \sum_{i=1}^N \zeta_i ,
	\end{split}
	\end{equation} \begin{equation}\label{eq:ode1}
	\begin{split}
	\frac{d\zeta_j(t)}{dt}  & = C_m  \sin\left(\frac{m}{4}\zeta_j\right)\sum_{l=1}^{N} \mathrm{sgn}(j-l)  \sin\left(\frac{m}{4}\zeta_l\right) \cos\left( \frac{m}{4}\left( 2(\beta_j -\beta_l) +  (\zeta_j - \zeta_l) \right) \right) 
	\end{split}
	\end{equation} and \begin{equation}\label{eq:ode2}
	\begin{split}
	\frac{d\gamma_{j+ \frac{1}{2}}(t) }{dt}& =   C_m \sin\left(\frac{m}{4}\gamma_{j+\frac{1}{2}}\right)   \sum_{l=1}^{N}\mathrm{sgn}(j+\frac{1}{2}-l)  \sin\left(\frac{m}{4}\zeta_l\right) \cos\left( \frac{m}{4}\left( (\beta_{j+1} -\beta_l) +  (\beta_j -\beta_l)  +  (\zeta_j - \zeta_l) \right) \right) 
	\end{split}
	\end{equation} for some constants $C_m, C_m' > 0$ and $C_m'' \ge 0$ depending only on $m$.  
\end{customthm} 

The last condition on $\Omega_0$ in the above says that up to a diffeomorphism, $\Omega_0$ is a union of non-intersecting sectors meeting at the origin locally; for a nice example, see Figure \ref{fig:petal} where $m = 3$. The proof of this result spans Sections 2--4 from \cite{SVP1}, and here we just use it as a black box to deduce interesting dynamics of vortex patches. 

\begin{figure}
	\includegraphics[scale=0.4]{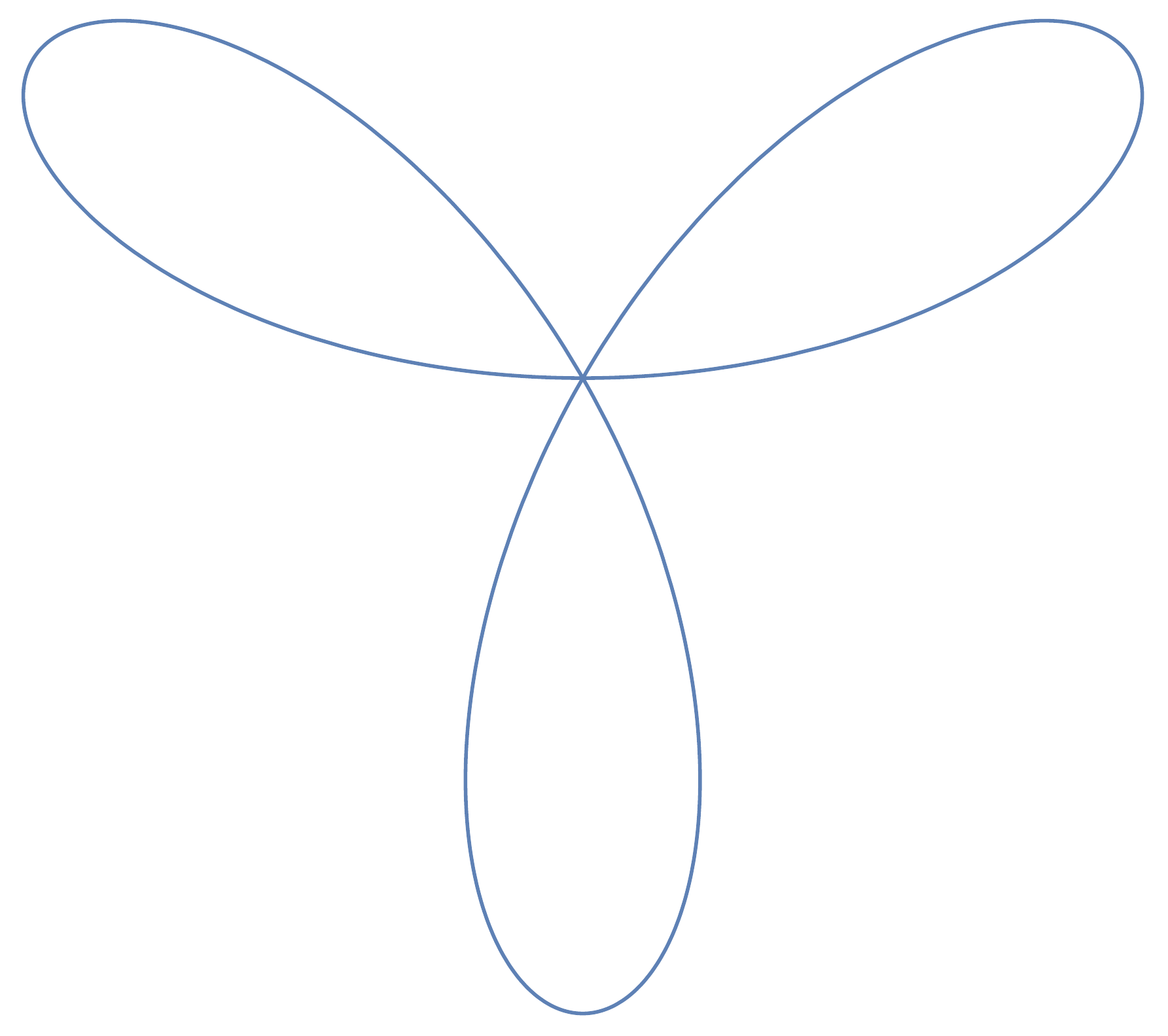} 
	\centering
	\caption{A petal domain}
	\label{fig:petal}
\end{figure} 

\begin{remark}
	A few remarks are in order. 
	\begin{itemize}
		\item Specifying $2N$ variables $\{ \beta_i, \zeta_i \}_{i=1}^N$ is equivalent to specifying $\beta_1, \zeta_1,\cdots, \zeta_N, \gamma_{1+\frac{1}{2}}, \cdots, \gamma_{N-\frac{1}{2}}$, via the relations \begin{equation*}
		\begin{split}
		\beta_j - \beta_l = \left(\gamma_{j-\frac{1}{2}} + \cdots + \gamma_{l+\frac{1}{2}}\right) + \left( \zeta_{j-1} + \cdots + \zeta_l \right), \qquad j > l .
		\end{split}
		\end{equation*}
		\item The absolute constants $C_m, C_m',$ and $C_m''$ can be determined as follows: when the vorticity exactly has the form \begin{equation*}
		\begin{split}
		\omega(r,\theta) = h(\theta)
		\end{split}
		\end{equation*} in polar coordinates with some profile $h$ depending only on the angle satisfying $h(\theta) = h(\theta + \frac{2\pi}{m})$ for some $m \ge 3$, the angular part of the corresponding velocity is simply $2rH(\theta)$ in polar coordinates, where \begin{equation*}
		\begin{split}
		4H + H'' = h. 
		\end{split}
		\end{equation*} The kernel expression for this elliptic problem can be found: \begin{equation*}
		\begin{split}
		H(\theta) = \int_{S^1} (c_m \sin(\frac{m}{2}|\theta - \theta'|) - c_m' )h(\theta')d\theta' 
		\end{split}
		\end{equation*} for some constants $c_m,c_m' $. 
		\item It follows from that the patch boundary $\partial\Omega(t)$ remains uniformly $C^{1,\alpha}$ up to the origin for all times. Indeed, in the proof of Theorem \ref{mainthm:wellposedness}, one has to propagate the following important pieces of information: for any finite $T>0$, $u(t) = K * {\bf{1}}_{\Omega(t)}$ satisfy \begin{equation*} 
		\begin{split}
		& \V \nabla u(t) \V_{L^\infty([0,T];L^\infty(\mathbb{R}^2))} < C(T)
		\end{split}
		\end{equation*} and \begin{equation*} 
		\begin{split}
		& \V \nabla u(t) \V_{L^\infty([0,T];C^\alpha(\overline{\Omega(t)}))} < C(T)
		\end{split}
		\end{equation*} for some constant $C(T)$. That is, the velocity field is globally uniformly Lipschitz and  $C^{1,\alpha}$ in the interior of the patch, uniformly up to the boundary. 
		\item We emphasize that the angles are determined for all $t > 0$ not just up to an overall constant (this fact is expressed in \eqref{eq:ode0}). 
		\item It is not difficult to see from \eqref{eq:ode1} that the sum $\zeta_1 + \cdots + \zeta_N$ is constant in time. This can be also derived from the conservation of circulation along a small loop containing the origin. 
		\item In the simplest case of $N = 1$ in \eqref{eq:local_expression}, the corners rotate with a constant angular speed for all time, which is determined only by the initial angle and $m$. 
		\item While strictly speaking the above result does not allow some (or all) angles to be zero, an analogous global well-posedness result can be shown even in such a setting; see \cite[Section 4.5]{SVP1}. Of course, in that case, angles which are initially zero (i.e. cusp) must remain so for all times. 
	\end{itemize}
\end{remark}

\section{Rotating patches}\label{sec:rot}


We demonstrate in this section that the system \eqref{eq:ode0}--\eqref{eq:ode2} gives rise to a large set of periodic (rotating) patch solutions to the 2D Euler equation. To be clear, we say that a solution $\omega(t)$ to 2D Euler is \textit{rotating} if there exists some nonzero constant $c $ such that $\omega(t,r,\theta) = \omega_0(r, \theta - ct)$. In particular it is a time-periodic solution. This is achieved by considering patches supported on multiple sectors with different strengths. For the simplicity of computations we shall focus on the case $m = 4$. Under this assumption, the sectors can be identified with intervals on $[-\frac{\pi}{4},\frac{\pi}{4})$ with endpoints identified by quotienting out the symmetry. Assume further that the $i$-th interval has strength $A_i$, which is an invariant of the 2D Euler equation. Then, the system of motion takes the following form, modulo a multiplicative constant: 
\begin{equation}\label{eq:ode1-4}
\begin{split}
\frac{d\zeta_j(t)}{dt}  & = \sin\left(\zeta_j\right)\sum_{l=1}^{N} \mathrm{sgn}(j-l) A_l \sin\left(\zeta_l\right) \cos\left( 2(\beta_j -\beta_l) +  (\zeta_j - \zeta_l) \right) 
\end{split}
\end{equation} and \begin{equation}\label{eq:ode2-4}
\begin{split}
\frac{d\gamma_{j+ \frac{1}{2}}(t) }{dt}& =    \sin(\gamma_{j+\frac{1}{2}})   \sum_{l=1}^{N}\mathrm{sgn}(j+\frac{1}{2}-l) A_l \sin\left(\zeta_l\right) \cos\left(   (\beta_{j+1} -\beta_l) +  (\beta_j -\beta_l)  +  (\zeta_j - \zeta_l) \right) . 
\end{split}
\end{equation} 
Further restricting to the case $N = 2$, and by assuming that the angles $\zeta_1, \zeta_2$, and $\gamma_{1+\frac{1}{2}}$ are stationary, we obtain \begin{equation}\label{eq:ode-2-stat}
\left\{
\begin{aligned}
0 & = -A_2 \sin(\zeta_1)\sin(\zeta_2) \cos(2\gamma + \zeta_1 + \zeta_2) \\
0 & = A_1 \sin(\zeta_1)\sin(\zeta_2) \cos(2\gamma + \zeta_1 + \zeta_2) \\
0 & = \sin(\gamma)\left( A_1 \sin(\zeta_1)\cos(\gamma+\zeta_1) - A_2 \sin(\zeta_2) \cos(\gamma+\zeta_2) \right)
\end{aligned}
\right. 
\end{equation} Assuming that $0 < \gamma, \zeta_1, \zeta_2 < \frac{\pi}{2}$ (we also have $\gamma+\zeta_1+\zeta_2<\frac{\pi}{2}$), it is straightforward to see that \eqref{eq:ode-2-stat} holds if and only if \begin{equation}\label{eq:sol} 
\begin{split}
& 2\gamma + \zeta_1+\zeta_2 = \frac{\pi}{2}, \quad A_1 \sin(\zeta_1)\cos(\gamma+\zeta_1) = A_2 \sin(\zeta_2) \cos(\gamma+\zeta_2)
\end{split}
\end{equation} Therefore we have a three-dimensional solution set in the five-dimensional phase space. We have arrived at the following proposition. \begin{proposition}
	In the case $N = 2$, $\omega(r,\theta) = \sum_{i=1}^2 \sum_{k=0}^{3} A_i {\bf 1}_{ \beta_i + \frac{\pi k}{4} }$ defines a rotating solution to 2D Euler if and only if \eqref{eq:sol} holds. We have particular solutions \begin{equation*} 
	\begin{split}
	&\zeta_1 = \frac{\pi}{8} + \xi, \quad \zeta_2 = \frac{\pi}{8} - \xi, \quad \gamma = \frac{\pi}{8},\quad A_1 = \sin(\frac{\pi}{8} - \xi)\cos(\frac{\pi}{4}-\xi), \quad  A_2 = \sin(\frac{\pi}{8}+\xi)\cos(\frac{\pi}{4}+\xi)
	\end{split}
	\end{equation*} for any $0 \le \xi \le \frac{\pi}{8}$. In particular, there exist rotating patches with zero mean: \begin{equation*} 
	\begin{split}
	& \int_{S^1} \omega \, d\theta = 0. 
	\end{split}
	\end{equation*}
\end{proposition}
The fact that there is a constant speed of rotation with zero mean sounds counter-intuitive. However, one should keep in mind that individual ``fluid particles'' are rotating at their own speed (constant in time), which altogether averages out to zero. 

It seems difficult to completely classify the set of periodic patches especially when the coefficients $A_i$ are allowed and $N $ is large. Still we have the following result which says that the support of a rotating solution cannot be too localized.  We recall that  (see \cite{EJ1,SVP1}), the solution of $$4H + H''= h$$ is simply given by \begin{equation}\label{eq:sol-H}
\begin{split}
H(\theta) = \frac{\pi}{8}\int_{-\frac{\pi}{4}}^{\frac{\pi}{4}} \left|\sin(2(\theta-\theta'))\right| h(\theta')d\theta'.
\end{split}
\end{equation} and that $2H$ is the speed of rotation. (Of course, the exact constant $\frac{\pi}{8}$ does not make a crucial difference.)
\begin{proposition}
	Let $\mathcal{I}$ be a disjoint union of intervals contained in $[-\frac{\pi}{8},\frac{\pi}{8}]$ and assume that $\omega_0 =   {\bf 1}_\mathcal{I} $ defines a rotating solution. Then $\omega_0 = {\bf 1}_{[b,a]}$ for some $-\frac{\pi}{8} \le b \le  a \le \frac{\pi}{8}$. 
\end{proposition}
\begin{proof}
	Consider the maximal value of $0 \le a \le \frac{\pi}{8}$ where $\mathcal{I} \subset [-a,a]$. Without loss of generality, we may assume that $\bar{\mathcal{I}}$ contains the point $a$, and let $I = [b,a]$ be the connected component of $\mathcal{I}$ containing $a$. Then, from \eqref{eq:sol-H} we obtain that \begin{equation*}
	\begin{split}
	H(a) - H(b) = \frac{\pi}{8} \int_{\mathcal{I} \backslash I} \sin(2(a-\theta)) - \sin(2(b-\theta)) d\theta 
	\end{split}
	\end{equation*} and note that $\sin(2x)$ is strictly increasing for $x \in [0,\frac{\pi}{4}]$. From the assumption on $\mathcal{I}$, $0 < a-\theta < \frac{\pi}{4} $ for all $\theta \in \mathcal{I}\backslash I$, and therefore $H(a) - H(b) = 0$ if and only if $\mathcal{I} = I$. 
\end{proof}
\begin{remark}
	The support condition is sharp; when the support is allowed to lie on $[-\frac{\pi}{8}-\epsilon, \frac{\pi}{8}+\epsilon]$ (say), then $\omega = {\bf 1}_{[-\frac{\pi}{8}-\epsilon,-\frac{\pi}{8}+\epsilon]} + {\bf 1}_{[\frac{\pi}{8}-\epsilon,\frac{\pi}{8}+\epsilon]}$ defines a rotating solution. 
\end{remark}
\begin{remark} The above argument covers the case when $\mathcal{I} = \cup_{i=1}^N I_i$ and $\omega = \sum_{i=1}^N A_i{\bf 1}_{I_i} $ with $A_i > 0$. 
	Moreover, one can show along the same lines that there is no time-periodic solution (up to rotation) whose support (modulo 4-fold symmetry) is contained in an interval of size $\frac{\pi}{4}$ for all time. 
\end{remark}

In the next section, we investigate the dynamics under the same assumption that $m = 4$ and $N = 2$. It turns out that any initial data converges to a rotating solution. As we shall see, in the special case $A_1 = A_2$, two angles generically converge to a single angle as $t \rightarrow +\infty$ which implies cusp formation. 

\section{Cusp formation in infinite time}\label{sec:cusp}


To demonstrate that there is some non-trivial long time dynamics of sectors, we consider the case when there are two sectors in a fundamental domain, assuming 4-fold rotational symmetry. Let us define $\Omega_1 \subset [0,\frac{\pi}{2})$ be the set obtained from $\Omega$ by quotienting out the rotational symmetry, and further write $\Omega_1 = \Omega_1^1 \cup \Omega_1^2$ (see Figure \ref{fig:two_corners}), and in this case, up to a rotation of $\mathbb{R}^2$, we only need to specify three angles; angles of each $\Omega_1^j$ and the angle in between. We denote them by $\zeta_1$, $\zeta_2$, and $\gamma$, respectively (see Figure \ref{fig:two_corners}). Then, the systems of equations \eqref{eq:ode1}, \eqref{eq:ode2} reduce to, up to a multiplicative constant which we neglect, \begin{equation}\label{eq:ode_reduced}
\left\{
\begin{aligned}
\dot{\zeta_1} &= - \sin(\zeta_1) \sin(\zeta_2) \cos(2\gamma + \zeta_1 + \zeta_2), \\
\dot{\zeta_2} &= \sin(\zeta_2)\sin(\zeta_1) \cos(2\gamma + \zeta_1 + \zeta_2), \\
\dot{\gamma} &= \sin(\gamma) \sin(\zeta_1 - \zeta_2) \cos(\gamma + \zeta_1 + \zeta_2).
\end{aligned}
\right.
\end{equation}  (Observe that $\zeta_1 + \zeta_2$ is an invariant of motion, as it should be). 

\begin{figure}
	\includegraphics[scale=1.0]{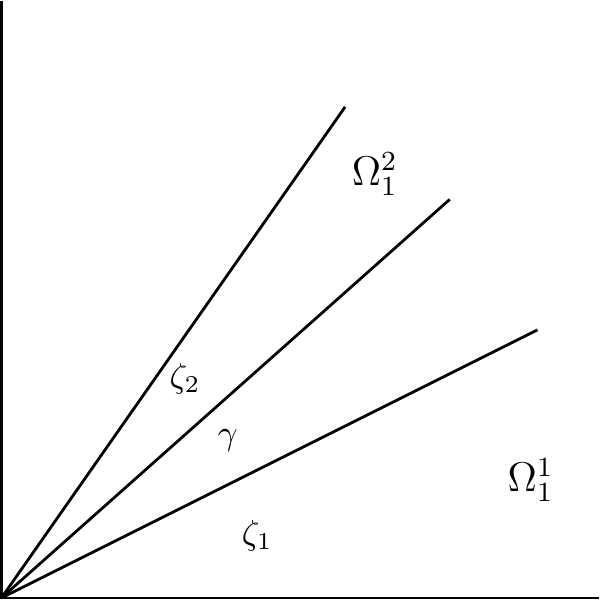} 
	\centering
	\caption{Evolution of two angles}
	\label{fig:two_corners}
\end{figure}

It can be easily shown that for exact infinite sectors, the whole patch $\Omega$ defines a purely rotating state if and only if $\zeta_2 = \zeta_1$ and $\gamma = \frac{\pi}{4} - \zeta_1$ (in which case $\Omega$ is indeed 8-fold symmetric) or one of the sectors is degenerate, that is, either $\zeta_2 = 0$ or $\zeta_1 = 0$. In this restricted setting, it can be established that any state converges as $ t\rightarrow +\infty$ to such a purely rotating space, and generically to a state where one of the angles become zero. 

Indeed, fix $\zeta_1(0) + \zeta_2(0) = \frac{\pi}{4}$, and moreover $\gamma(0) = \zeta_2(0)$. Then using the system \eqref{eq:ode_reduced}  one sees that $\gamma = \zeta_2 $ for all time. Alternatively, assuming such an initial data, subtracting $-\frac{1}{2}$ from vorticity everywhere in the plane gives an odd configuration with respect to the line separating $\zeta_2$ and $\gamma$ (see Figure \ref{fig:two_corners}). Since the Euler equations preserve odd symmetries of vorticity, it follows that $\gamma= \zeta_2$ for all time. Therefore, the one-dimensional system we get is: \begin{equation}\label{eq:ode_1D}
\begin{split}
\dot{\gamma} = \sin(\gamma)\sin(\frac{\pi}{4} - 2\gamma) \cos(\gamma + \frac{\pi}{4}). 
\end{split}
\end{equation} On the other hand, if we keep the assumption $\zeta_1(0)+ \zeta_2(0) = \frac{\pi}{4}$ but now take $\gamma(0) = \zeta_1(0)$, we instead obtain \begin{equation}\label{eq:ode_1D_prime}
\begin{split}
\dot{\gamma} = -\sin(\gamma)\sin(\frac{\pi}{4} - 2\gamma) \cos(\gamma + \frac{\pi}{4}). 
\end{split}
\end{equation} Then, one sees that for any initial value $0 < \gamma(0) < \frac{\pi}{4}$, it always converges to $\pi/8$ for $t \rightarrow + \infty$ in the former case, and either $0$ to $\frac{\pi}{4}$ in the latter. Actually, under the constraint $\zeta_1 + \zeta_2 = \frac{\pi}{4}$, the former case is the only situation where the forward asymptotic state is 8-fold symmetric, not just 4. See a plot of the phase portrait in Figure \ref{fig:phase_portrait}. 

\begin{figure}
	\includegraphics[scale=0.4]{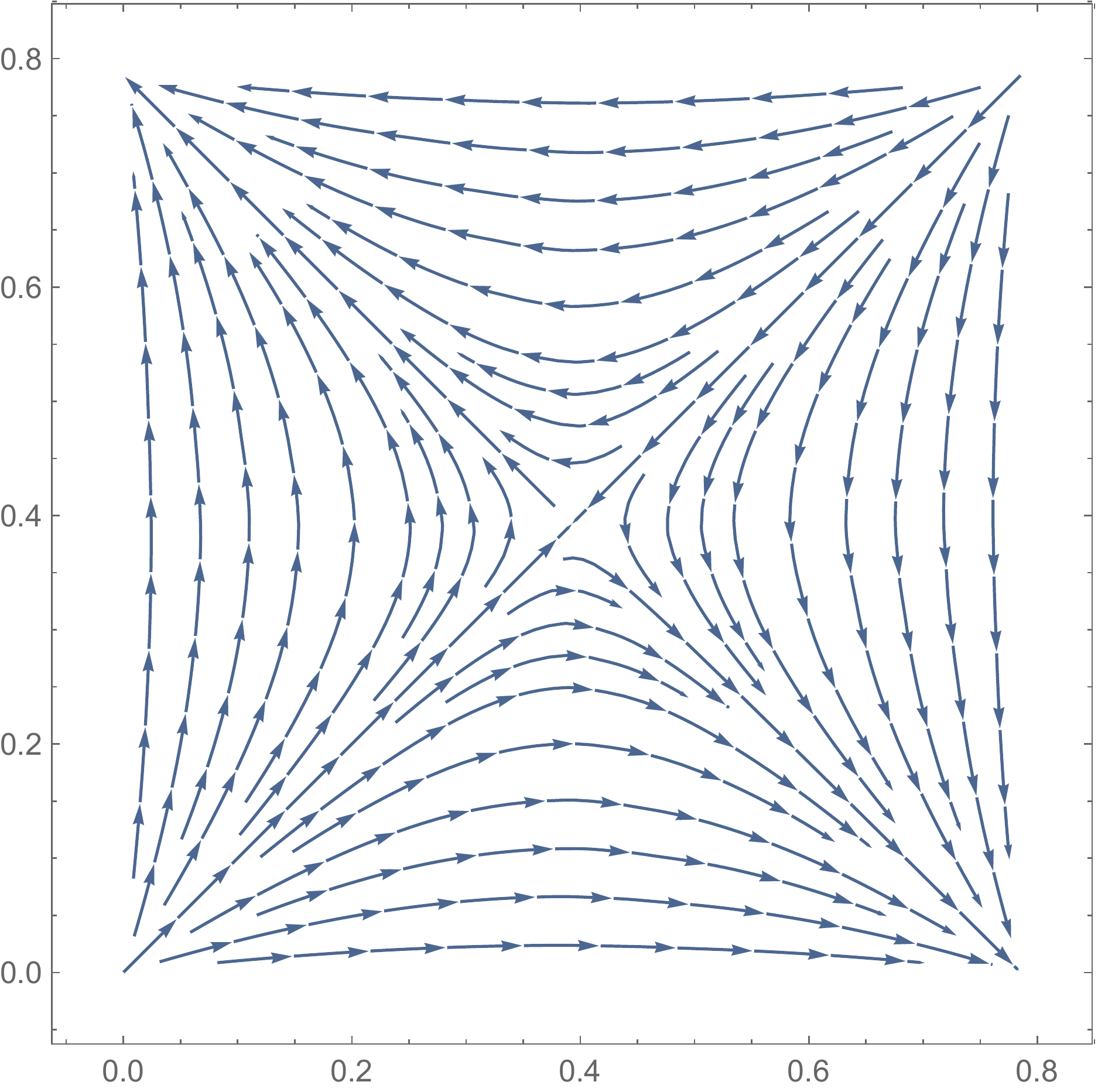} 
	\centering
	\caption{The phase portrait of the system \eqref{eq:ode_reduced} under the assumption $\zeta_1 = \frac{\pi}{4} - \zeta_2$. The axis correspond to variables $\zeta_2$ and $\gamma$ taking values in $[0,\frac{\pi}{4}]$, which play symmetric roles.}
	\label{fig:phase_portrait}
\end{figure}

Thanks to Theorem \ref{mainthm:wellposedness}, we have shown the following \begin{theorem}\label{thm:cusping}
	Assume that $\Omega_0$ is a 4-fold symmetric $C^{1,\alpha}$-patch with two angles $0 < \zeta_1(0)$ and $0 < \zeta_2(0)$ separated by an angle $0 < \gamma$, in a fundamental domain. Further assume that $\zeta_1(0) + \zeta_2(0) = \frac{\pi}{4}$ and $\gamma(0) \ne \zeta_2(0)$. Then, depending on whether $\zeta_2(0) > \gamma(0)$ or $\gamma(0) > \zeta_2(0)$ holds, we have $\zeta_1(t) \rightarrow 0$ or $\zeta_2(t) \rightarrow 0$, respectively, as $t \rightarrow + \infty$. That is, one of the two components of $\Omega_t$ in each fundamental domain cusps in infinite time. The speed of cusp formation is exponential in time. 
\end{theorem}
\begin{proof}
	We only need to show that the angle collapses with an exponential rate. Without loss of generality, assume that $\zeta_1 \rightarrow 0$ as $t \rightarrow +\infty$. Then, we have $\gamma \rightarrow 0$ as well, and the equation \eqref{eq:ode_reduced}  becomes \begin{equation*}
	\begin{split}
	\dot{\zeta_1} \approx -C\zeta_1 
	\end{split}
	\end{equation*} for some positive constant $C > 0$. This finishes the proof. 
\end{proof}

It is likely that for any finite number of sectors, with constant vorticity, there is asymptotic convergence of a purely rotating state.

\section{Spiral formation in infinite time}\label{sec:spiral}

In this subsection, we construct a patch solution which develops an infinite spiral  as time goes to infinity. To be more precise, the winding number of the boundary of our patch solution around the origin increases at least linearly with time. This will be achieved by combining the well-known stability result for the  circular vortex patch with the persistence of the corner angle (Theorem \ref{mainthm:wellposedness} in the case $N = 1$). The idea is very simple: one may perturb the circular patch near the center so that the patch locally looks like a symmetric union of sectors. Then, an explicit computation shows that the rotation speed at the corner is different from the bulk rotation speed, which forces a spiral to form linearly in time. In the following we make this idea precise. 


We begin with the statement of our result. 
\begin{theorem}\label{thm:spiral}
	There exists a vortex patch $\Omega_0$ of compact support, with boundary $C^\infty$ away from the origin, whose solution $\Omega(t)$ spirals linearly in time as $t \rightarrow +\infty$. More precisely, for each $t \ge 0$, there is an injective curve $\gamma(t) \subset \partial\Omega(t)$ whose winding number around the origin is at least $ct$ with some constant $c > 0$ depending only on $\Omega_0$. In particular, for any line $L$ passing through the origin, we have \begin{equation*}
	\begin{split}
	\forall t \ge 0, \quad |\{ x \in \mathbb{R}^2 : x \in \partial\Omega(t) \cap L \} | \ge ct - c' 
	\end{split}
	\end{equation*} for some $c' > 0$. In addition, at least one of the following options hold: (i) the perimeter of the patch goes to infinity:  \begin{equation*}
	\begin{split}
	\limsup_{t \rightarrow \infty} |\partial\Omega(t)| = + \infty, 
	\end{split}
	\end{equation*} (ii) the turns accumulate at the origin: \begin{equation*}
	\begin{split}
	\forall r > 0, \quad \lim_{t \rightarrow \infty} |\{ |x| \le r  : x \in \partial\Omega(t) \cap L \} | = +\infty
	\end{split}
	\end{equation*} for any line $L$ containing the origin.
\end{theorem}

\begin{example}
	When the domain of the fluid has boundary and a non-trivial fundamental group, spiral formation in infinite time is very easy to achieve. One can simply take the annulus $\mathbb{A} = \{ 1 \le r \le 2 \}$ to be the fluid domain and take a patch $\Omega_0 \subset \mathbb{A}$ which extends to both components of $\partial\mathbb{A}$. It is not difficult to arrange that the velocities on two component of $\partial\mathbb{A}$ are bounded away from each other for all times. This implies that the pieces of $\Omega(t)$ lying on different components of $\partial\mathbb{A}$ will rotate with different angular speeds for all times, which means linear in time spiral formation. Note that in this example the length of the patch boundary trivially goes to infinity as time goes to infinity. This example is essentially due to Nadirashvili \cite{Nad} who observed in this setup infinite in time growth of the vorticity gradient for smooth vorticities (instead of patches). 
\end{example}

\begin{example}
	Even in the case of the trivial fundamental group, when boundary is present, the boundary of a large chunk of vorticity could play the role of an essential curve. For concreteness take the square domain $[0,1]^2$ and place a patch $U_0$ which has area almost equal to 1 and does not touch the boundary of the square. Then now the region $[0,1]^2\backslash U_0$ is homeomorphic to an annulus. Take a patch $\Omega_0 \subset [0,1]^2\backslash U_0 $ whose boundary intersects $\partial([0,1]^2)$ and $\partial U_0$. We demand that the intersection with $\partial([0,1]^2)$ is contained in the segment $[0,1]\times \{ 0\}$ (such an example of $\Omega_0$ is depicted in Figure \ref{fig:spiral_square_initial}). One may add rotated images of $\Omega_0$ around the center $(\frac{1}{2},\frac{1}{2})$ to make the configuration 4-fold symmetric, which adds extra stability of the scenario. One can show that the central patch $U(t) = \Phi(t,U_0)$ (where $\Phi(t,\cdot)$ is the particle trajectory map associated with the initial patch data) rotates around the center infinitely many times as $t \rightarrow \infty$ (cf.  Kiselev-Sverak \cite{KS}). On the other hand, the part of the patch boundary touching $[0,1]\times\{0\}$ simply converges to the corner $(1,0)$. This guarantees spiral formation in infinite time. The length of the boundary of the patch again goes to infinity in this example. 
\end{example}

\begin{figure}
	\includegraphics[scale=0.6]{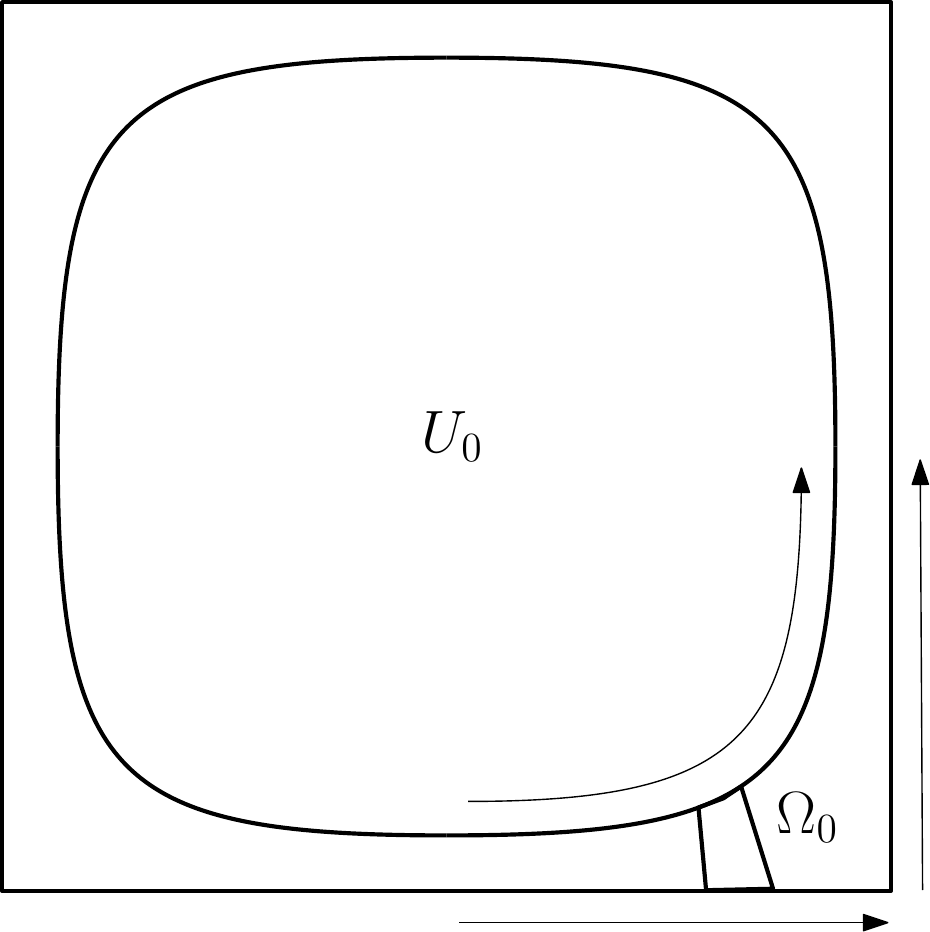} 
	\centering
	\caption{Spiral formation in the case of the square}
	\label{fig:spiral_square_initial}
\end{figure}

In view of the examples above, the main point of Theorem \ref{thm:spiral} is that we can achieve spiral formation in the absence of the boundary of the fluid domain. 

\begin{example}\label{ex:spiral}
	Here we give an explicit example of $\Omega_0$ with which the conclusion of Theorem \ref{thm:spiral} holds. The patch $\Omega_0$ will be given by the 4-fold symmetrization of a patch $\Omega_0^1$ which lies strictly inside the region $\{ |x_2| < x_1 \}$ except at the origin. For simplicity we use polar coordinates to define $\Omega_0^1$. For sufficiently small $\delta, \nu > 0$ and any $0 < \theta_0 < \frac{\pi}{4} - \nu$, take the points \begin{equation*}
	\begin{split}
	A^{\pm} = (\delta, \pm \theta_0), \quad B^{\pm} = (1, \pm(\frac{\pi}{4} - \nu) )
	\end{split}
	\end{equation*} in the $(r,\theta)$-coordinates. Then, draw straight lines between the origin $O$ and $A^+$, and between $A^+$ and $B^+$. Similarly connect $O$ and $A^-$, and $A^-$ and $B^-$ by straight lines. Finally, connect $B^+$ with $B^-$ by an arc which belongs to the unit circle centered at $O$. This gives a closed piecewise smooth curve which is depicted in Figure \ref{fig:spiral_initial}. Then we may smooth out the patch boundary locally near the points $A^\pm$, $B^\pm$ so that in the ball $B(0,\delta/2)$, $\Omega_0^1$ is still a sector and the patch boundary is $C^\infty$-smooth except at $O$. Then, simply set $\Omega_0 = \cup_{j = 0}^3 R_{j\pi/2} \Omega_0^1$. Note that by taking $\delta, \nu \rightarrow 0^+$, the area of  $\Omega_0$ converges to that of the unit circle. We shall use this example in the proof of Theorem \ref{thm:spiral}. 
\end{example}

\begin{figure}
	\includegraphics[scale=0.6]{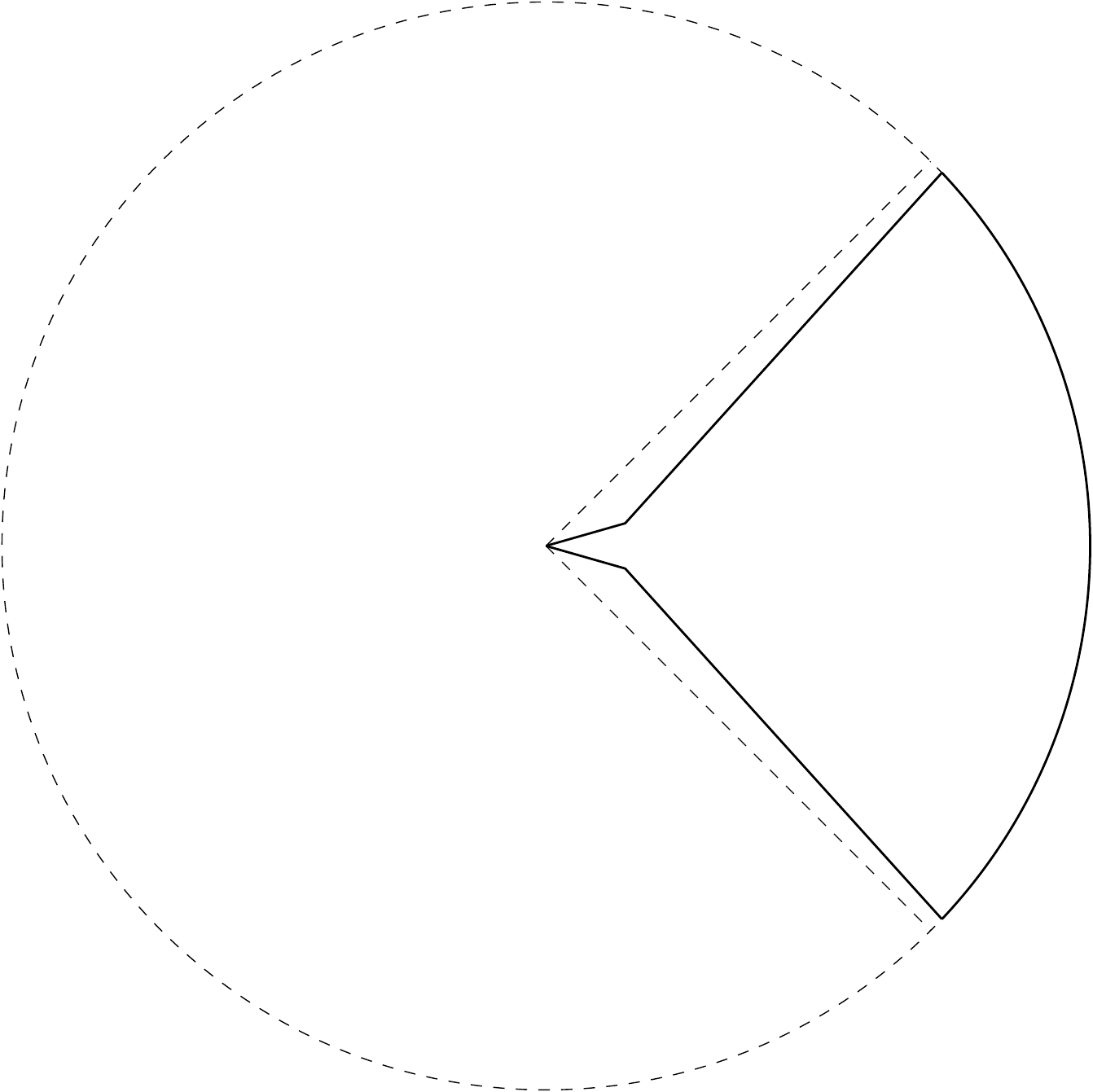} 
	\centering
	\caption{Initial data for the result of Theorem \ref{thm:spiral} (before smoothing)}
	\label{fig:spiral_initial}
\end{figure}

Before we give a proof of Theorem \ref{thm:spiral}, let us briefly review the $L^1$-stability theorem for the circular vortex patch first proved by Wan and Pulvirenti \cite{WP}.\footnote{Strictly speaking, the result in this paper is stated for patches contained inside a ball but the authors mention that the domain can be replaced by $\mathbb{R}^2$.} Here we state the version given later by Sideris and Vega \cite{SiVe}. We also mention classical (but weaker) stability results for the circular patch by Dritschel \cite{Dr2} and Saffman \cite{Saff}. 

\begin{theorem*}[Sideris and Vega \cite{SiVe}]
	Let ${\bf 1}_B$ be the characteristic function supported on the unit ball inside $\mathbb{R}^2$. For any bounded open set $\Omega_0 \subset \mathbb{R}^2$, we have \begin{equation*}
	\begin{split}
	\V {\bf 1}_{\Omega(t)} - {\bf 1}_B \V_{L^1}^2 \le 4\pi \sup_{ \Omega_0 \triangle B } |1 - |x|^2| \V {\bf 1}_{\Omega_0} - {\bf 1}_B\V_{L^1}
	\end{split}
	\end{equation*} for all $t \ge 0$, where $\Omega(t)$ is the patch solution in $\mathbb{R}^2$ associated with initial data $\Omega_0$. 
\end{theorem*}

We shall now define the notion of winding number for curves satisfying $\gamma : [0,a] \rightarrow \mathbb{R}^2$, $0\notin \gamma((0,a])$,  and $\gamma(0) = 0$ with well-defined tangent vector at $0$. Then, we may consider the continuous map $h_\gamma : (0,a] \rightarrow S^1$ defined for $b \in (0,a]$ by taking the angle of $\gamma(b) \in \mathbb{R}^2$ in polar coordinates. Here, $S^1$ denotes the interval $[0,2\pi]$ with the endpoints identified, so that there is a natural projection map $\pi : \mathbb{R} \rightarrow S^1$. The map $h_\gamma$ has a unique extension (up to an additive constant) to a continuous map $\tilde{h}_\gamma : [0,a] \rightarrow \mathbb{R}$ satisfying $\pi \circ \tilde{h}_\gamma = h_\gamma$. Here, we may define $\tilde{h}_\gamma$ continuously at 0 using the assumption that $\gamma$ has a well-defined tangent vector at $0$. Then, we define the winding number of $\gamma$ by $N[\gamma] = (\tilde{h}_\gamma(a) - \tilde{h}_\gamma(0))/2\pi$.


We shall recall the following elementary \begin{lemma}\label{lem:winding}
	Let $\gamma : [0,a] \rightarrow \mathbb{R}\backslash \{ 0 \}$ be a continuous curve not touching the origin except at $0$ and has a well-defined tangent at 0. Then, for any line $L$ passing through the origin, \begin{equation*}
	\begin{split}
	\left| \{ x : x \in L \cap \gamma([0,a]) \} \right| \ge \lfloor N[\gamma]\rfloor 
	\end{split}
	\end{equation*} where $ \lfloor N \rfloor$ denotes the largest integer not exceeding $N$. 
\end{lemma} 
We omit the proof, which is a simple application of the intermediate value theorem (cf. Figure \ref{fig:winding}). 

\begin{figure}
	\includegraphics[scale=0.6]{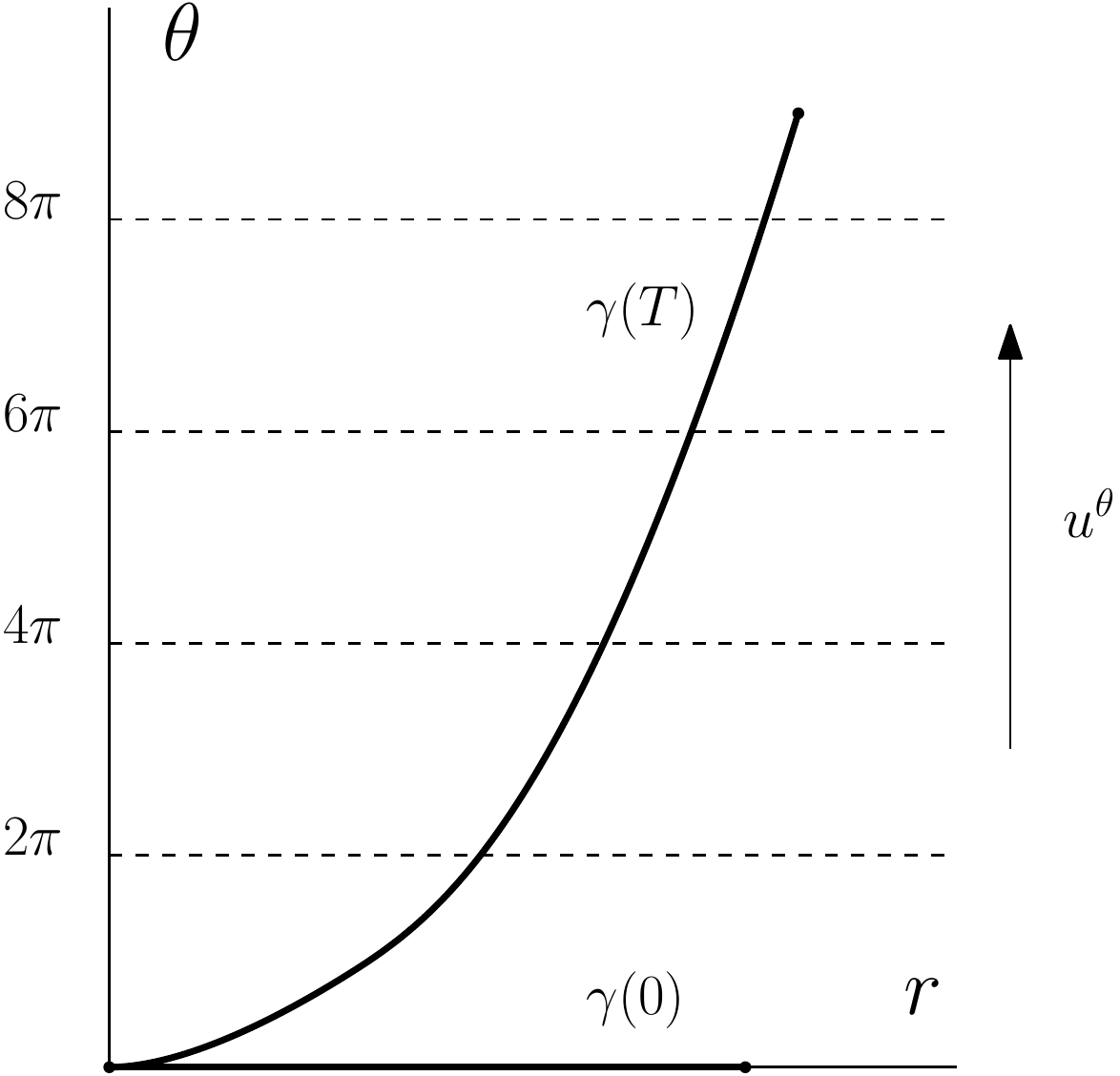} 
	\centering
	\caption{The winding number for $\gamma(t)$ is defined by the $\theta$-coordinate of the right endpoint in $(r,\theta)$ coordinates, where $\theta$ is now varying on $\mathbb{R}$ instead of $S^1$. A lower bound on $u^\theta(t,\cdot)$ at the right endpoint of $\gamma(t)$ guarantees increase of the winding number with $t$. Moreover, a large winding number guarantees a large number of intersections with any line in $\mathbb{R}^2$ passing through the origin.}
	\label{fig:winding}
\end{figure} 

We are ready to give a proof of Theorem \ref{thm:spiral}. The idea is simply to take $\gamma$ to be a curve on $\partial\Omega_0$ starting at the origin and ending at a point in $\partial\Omega_0 \cap \partial B$ where $B$ is the unit disc. Then both endpoints rotate with different angular speed (strictly speaking the origin is fixed but points on $\partial\Omega(t)$ arbitrarily close to the origin rotates at a constant angular speed by persistence of the corner angle), causing the winding number to grow (linearly) with time. A slight complication arises since the initial endpoint of such a curve may slowly move towards the origin as the patch evolves, which then may not rotate around the origin with the desired angular speed.  

\begin{proof}[Proof of Theorem \ref{thm:spiral}]
	We fix some $0 < \theta_0 < \frac{\pi}{4}$ and for any given $\epsilon > 0$, we may take the 4-fold symmetric patch $\Omega_0$ as in Example \ref{ex:spiral} with $\delta, \nu > 0$ sufficiently small, so that the stability theorem gives \begin{equation*}
	\begin{split}
	\V {\bf 1}_{\Omega(t)} -{\bf 1}_B \V_{L^1} \le \epsilon^2\quad\forall t \ge 0 
	\end{split}
	\end{equation*} where $\Omega(t)$ is the patch solution with initial data $\Omega_0$. Below we shall choose $\epsilon > 0$ to be sufficiently small with respect to several parameters depending only on $\theta_0$. 
	
	Assume for a moment that, for any $t \ge 0$, there is an injective curve $\gamma(t) \subset \partial\Omega(t)$ which has winding number greater than $ct$ for some constant $c = c(\theta_0) > 0$. Using this, the rest of the statements of Theorem \ref{thm:spiral} follows immediately. The statement regarding the number of intersections follows from Lemma \ref{lem:winding}. To see that the second statement holds, assume that the option (ii) does not hold; that is, there exists some $r > 0$, \begin{equation*}
	\begin{split}
	\lim_{ t \rightarrow \infty} | \{ |x| < r : x \in \partial\Omega(t) \cap L \} | \ne +\infty. 
	\end{split}
	\end{equation*} Then there exists a sequence of time moments $t_k \rightarrow +\infty $ such that \begin{equation*}
	\begin{split}
	| \{ |x| < r : x \in \partial\Omega(t_k) \cap L \} | \le M 
	\end{split}
	\end{equation*} for some $M > 0$. For each $t_k$, we deduce from $N[\gamma(t_k)] \ge ct_k$ and the above that whenever $k$ is sufficiently large, $\partial\Omega(t_k) \ge c'rt_k$ for some constant $c' > 0$. This establishes (i). 
	
	Returning to the proof of the above claim, let $u(t)$ and $u_B$ be the velocities associated with patches ${\bf 1}_{\Omega(t)}$ and ${\bf 1}_B$, respectively. Then, from \begin{equation*}
	\begin{split}
	u(t,x) - u_B(x) = \frac{1}{2\pi} \int_{\mathbb{R}^2} \frac{(x-y)^\perp}{|x-y|^2} \left( {\bf 1}_{\Omega(t)}(y) - {\bf 1}_{B}(y) \right) dy \end{split}
	\end{equation*} we compute that, after splitting $\mathbb{R}^2$ into regions $\{ |x-y| < r \}$ and $\{ |x-y| \ge r\}$ and then choosing $r$ to make two terms on the right hand side equal, \begin{equation*}
	\begin{split}
	|u(t,x) - u_B(x)| & \le \frac{1}{2\pi}\left( r \V {\bf 1}_{\Omega(t)}- {\bf 1}_{B}\V_{L^1} + r^{-1} \V {\bf 1}_{\Omega(t)} - {\bf 1}_{B} \V_{L^\infty} \right) \\
	& \le \frac{1}{\pi} \V {\bf 1}_{\Omega(t)}- {\bf 1}_{B}\V_{L^1}^{\frac{1}{2}} \V {\bf 1}_{\Omega(t)}- {\bf 1}_{B}\V_{L^\infty}^{\frac{1}{2}}
	\end{split}
	\end{equation*} so that we conclude \begin{equation}\label{eq:vel_L_infty}
	\begin{split}
	\V u(t) - u_B \V_{L^\infty(\mathbb{R}^2)} \le  \epsilon
	\end{split}
	\end{equation} for all $t \ge 0$. 
	
	We now compute explicitly the angular velocity of the patch boundary at the corner, which remains the same for all times. At the initial time, we may write \begin{equation*}
	\begin{split}
	u_{\Omega_0} = u_S + \tilde{u},
	\end{split}
	\end{equation*} where $u_{\Omega_0}$ and $u_S$ are velocities associated with patches $\Omega_0$ and $S = \cup_{j=0}^3 \{ (r,\theta) : -\theta_0 +j\pi/2 < r < \theta_0 + j\pi/2 \}$. Since the patches $\Omega_0$ and $S$ coincide in a ball centered at 0, $|\tilde{u}(x)| \ll |x|$ as $|x| \rightarrow 0$. We then compute \begin{equation*}
	\begin{split}
	\lim_{r \rightarrow 0}\frac{(u_{\Omega_0}\cdot e^\theta)(r,\theta_0)}{r} = \frac{(u_{S}\cdot e^\theta)(r,\theta_0)}{r}  &= \int_0^{2\pi} |\sin(2(\theta_0 - \theta'))| {\bf 1}_{[-\theta_0,\theta_0]} d\theta' \\
	&= \frac{1}{4}(1 - \cos(4\theta_0)).
	\end{split}
	\end{equation*} From the ODE system of \ref{mainthm:wellposedness}, \begin{equation*}
	\begin{split}
	\forall t > 0, \quad \lim_{r \rightarrow 0}\frac{1}{r}(u_{\Omega(t)}\cdot e^\theta)(r,\theta_0 +\frac{t}{4}(1 - \cos(4\theta_0)) ) = \frac{1}{4}(1 - \cos(4\theta_0)).
	\end{split}
	\end{equation*} Note that \begin{equation*}
	\begin{split}
	0 < \frac{1}{4}(1 - \cos(4\theta_0)) < \frac{1}{2},
	\end{split}
	\end{equation*} whereas $$\frac{u_B\cdot e^\theta}{r} = \frac{1}{2}$$ for $r \le 1$.
	
	For simplicity, let us set $c_0 = \frac{1}{4}(1 - \cos(4\theta_0))$ and work in a reference frame which rotates in the clockwise direction (note that the patches are rotating in the counter-clockwise direction since $e^\theta = (e^r)^\perp$) with angular speed $c_0$ around the origin, so that the tangent vectors at the corner are stationary for all times. In this frame, the unit disc now rotates with angular speed $c_1 := \frac{1}{2} - c_0$. 
	
	\medskip
	
	\noindent \textbf{Claim.} For any $T > 0$ there exists a point $x(T) \in \Omega(T)^1 := \Phi(T,\Omega_0^1)$ such that 
	the curve $\Phi(T,\gamma(T))$ has winding number at least $cT$ for some  constant $c > 0$, where $\gamma(T)$ is any injective curve belonging to the initial patch $\Omega_0^1$ and connecting the origin with the point $\Phi_T^{-1}(x(T)) =: \tilde{x}(T)$.
	
	\medskip 
	
	To show this, we observe that by taking $\epsilon > 0$ sufficiently small, an arbitrarily high portion of the points inside the patch rotates with an angular speed comparable to that of the unit disc. That is, for given small $r_0 > 0$, if $\rho \ge r_0$, then we have from \eqref{eq:vel_L_infty} that \begin{equation*}
	\begin{split}
	\V u(t)\cdot e^\theta - u_B \cdot e^\theta \V_{L^\infty} \le \epsilon,
	\end{split}
	\end{equation*} and evaluating it at a point $x$ with distance to the origin $\rho$, since $u_B\cdot e^\theta = c_1\rho$, we obtain \begin{equation*}
	\begin{split}
	|u^\theta(t,x) - c_1| \le \frac{\epsilon}{\rho} \le \frac{\epsilon}{r_0}. 
	\end{split}
	\end{equation*} That is, given $r_0 > 0$, we may take $\epsilon > 0$ smaller if necessary to guarantee that \begin{equation}\label{eq:lowerbound_av}
	\begin{split}
	u^\theta(t,x) \ge \frac{c_1}{2},\quad \forall t\ge 0, |x| \ge r_0. 
	\end{split}
	\end{equation} Assume that a point $\tilde{x} \in \Omega_0^1 $ has the property that  \begin{equation*}
	\begin{split}
	\frac{1}{T}\left| A \right| := \frac{1}{T}\left| \{ 0 \le  t \le T : |\Phi(t,\tilde{x})| \ge r_0 \} \right| \ge \eta ,
	\end{split}
	\end{equation*}  where $\eta>0$ is sufficiently close to 1. Then, taking $\gamma(T)\subset \overline{\Omega_0^1}$ to be an injective curve connecting the origin to $\tilde{x}$, we have that the winding number of the image $\Phi(T,\gamma(T))$ satisfies (cf. Figure \ref{fig:winding}) \begin{equation*}
	\begin{split}
	N[\Phi(T,\gamma(T))] = \int_0^T u^\theta(t, \Phi(t, \tilde{x}) ) dt &= \int_A u^\theta(t, \Phi(t, \tilde{x}) ) dt + \int_{[0,T]\backslash A} u^\theta(t, \Phi(t, \tilde{x}) ) dt \\
	&\ge \eta T \frac{c_1}{2} - C(1-\eta)T,
	\end{split}
	\end{equation*} where $C > 0$ is an absolute constant in the estimate \begin{equation*}
	\begin{split}
	|u^\theta(t,x)| := \frac{|u(t,x)\cdot e^\theta|}{|x|} \le C\V \omega(t)\V_{L^\infty} = C
	\end{split}
	\end{equation*} which holds under the 4-fold symmetry assumption on $\omega$. Here, we have used that \begin{equation*}
	\begin{split}
	\frac{d}{dt} N[\Phi(t,\gamma(t))] = u^\theta(t,\Phi(t,\tilde{x})),
	\end{split}
	\end{equation*} which is a direct consequence of the definition of the winding number. Once we have chosen $\eta$ sufficiently close to 1 that \begin{equation*}
	\begin{split}
	\frac{\eta c_1}{2} > \frac{C(1-\eta)}{2}
	\end{split}
	\end{equation*} holds, we conclude that \begin{equation*}
	\begin{split}
	N[\Phi(T,\gamma(T))] \ge \frac{\eta c_1}{4} T 
	\end{split}
	\end{equation*} which finishes the proof of the \textbf{Claim} with $x(T) := \Phi(T, \tilde{x})$. We shall now show existence of such a point. Assume towards contradiction that there exists some $0 < \eta < 1$ and $T > 0$ for which every point in $\Omega_0^1$ spends less than $\eta$-fraction of time outside the ball $B_0(r_0)$ during the time interval $[0,T]$. The total area of points from $\Omega(t)^1$ which can belong to $B_0(r_0)$ at any moment of time $0 \le t \le T$ is clearly bounded by $ \pi r_0^2 / 4$. Integrating over time, we get an upper bound of $\pi r_0^2T/4$. On the other hand, since every point of $\Omega_0^1$ is forced to spend at least $\eta$-fraction of time inside $B_0(r_0)$, we get a lower bound of $\eta T |\Omega_0^1|$ for the total area of points from $\Omega(t)^1$ lying in $B_0(r_0)$, integrated over the time interval $[0,T]$. We get a contradiction once \begin{equation*}
	\begin{split}
	\eta T |\Omega_0^1| > \frac{\pi r_0^2 T}{4},
	\end{split}
	\end{equation*} which is easy to arrange by taking $|\Omega_0^1|$ close to $\frac{\pi}{4}$ and $r_0 > 0$ small. 
	
	We are now in a position to finish the proof, by a simple continuity argument. Fix some $T > 0$ and take the point $\tilde{x}(T)$ satisfying the property described in the above \textbf{Claim}. There is a curve $\Phi(T,\gamma(T)) \subset \Omega(T)^1$ with winding number at least $cT$. Using this it is easy to find, with a continuity argument, an injective curve along the boundary $\partial(\Omega(T)^1)$, starting at the origin, with winding number exceeding that of $\Phi(T,\gamma(T))$. To see this we consider the curve $\tilde{\gamma}$ lying on $[0,\infty) \times S^1$ defined by expressing the image of $\gamma(T)$ in $\mathbb{R}^2$ with polar coordinates. Then using the canonical projection map $\mathbb{R} \rightarrow S^1$, we can lift this curve (its image, to be precise) to lie on $[0,\infty) \times \mathbb{R}$. Then, by definition of the winding number, $\tilde{\gamma}$ has a point in its image with second coordinate at least $cT$. Repeating this procedure for the boundary $\partial(\Omega(T)^1)$, viewed as a closed and injective curve starting and ending at the origin, we see that the corresponding image in $[0,\infty) \times \mathbb{R}$ should contain the image of $\tilde{\gamma}$. In particular, there is a point on this image with its second coordinate strictly larger than $cT$. The proof is now complete. 
\end{proof}

\bibliographystyle{plain}
\bibliography{thesis}

\end{document}